\documentclass{amsart}
\usepackage{amssymb,amsmath, amsthm,latexsym}
\usepackage{graphics}
\usepackage{todonotes}
\usepackage{amscd}
\usepackage{graphics}
\usepackage{here}
\newcommand{\cal}[1]{\mathcal{#1}}
\theoremstyle{plain}
\newtheorem{theo}{Theorem}

\newtheorem{lemma}{Lemma}[section]

\newtheorem{proposition}[lemma]{Proposition}
\newtheorem{corollary}[lemma]{Corollary}
 
\theoremstyle{definition}

\newtheorem{remark}[lemma]{Remark}
\parskip=\bigskipamount
\newcommand{\dist}{\operatorname{dist}}

\let\phialt=\phi
\let\phi=\varphi
\let\varphi=\phialt
\begin{document}
\title{Actions of finitely generated groups on compact metric spaces}
\author{Ursula Hamenst\"adt}
\date{August 30, 2023}

\begin{abstract}
Let $\Gamma$ be a finitely generated group which admits an action 
by homeomorphisms on a compact metrizable space $X$. 
We show that there is a metric on $X$ defining the original topology such 
that for this metric, the action is by bi-Lipschitz transformations. 
\end{abstract}

\maketitle

\section{introduction}

The homeomorphism group ${\rm Homeo}(S^1)$ 
of the circle $S^1$ contains many 
large and interesting subgroups, but there are also many rigidity 
results stating that large classes of groups do not embed
into ${\rm Homeo}(S^1)$. Most of the known rigidity results require
however
that the action is by diffeomorphisms of class at least $C^{1,\alpha}$ for some
$\alpha >0$, but it is generally believed that many of these results
are true in larger generality. 
We refer to the recent book \cite{KK21} for an account of 
what is known to date. 

On the other hand,  for \emph{countable} groups acting on the circle, 
some additional regularity can always be assumed. The following is 
Theorem D of \cite{DKN07}.

\begin{theo}[Deroin, Kleptsyn and Navas]\label{deroin} 
An action of a countable group $\Gamma$ on $S^1$ is 
conjugate to an action by bi-Lipschitz transformations.
\end{theo}

This result is sharp: There are homeomorphisms of higher dimensional
manifolds which are not conjugate to Lipschitz maps 
\cite{H79}. 
The goal of this note is to point out that from the point of view 
of rigidity of actions on compact metric spaces, Theorem \ref{deroin} 
is not specific to groups acting on $S^1$.

\begin{theo}\label{main} 
Let $\Gamma$ be a finitely generated group acting as a group of 
homeomorphisms on a compact metrizable space $X$. Then there
exists a metric $d$ on $X$ defining the original topology such that
the action is by bi-Lipschitz transformations for $d$.
\end{theo}

For finitely generated groups, Theorem \ref{deroin}
follow from Theorem \ref{main} as will be discussed at the end of this note.
It is in this conclusion where specific properties of the circle enter.
We do not know if
Theorem \ref{main} holds true for countable infinitely generated groups.

{\bf Acknowledgement:} I am grateful to
Sang-hyun Kim for bringing the article
\cite{DKN07} to my attention and for useful discussions.

\section{Proof of the theorem}\label{proof}

Consider a finitely generated group $\Gamma$.
Denote by ${\cal C}$ the Cayley graph of $\Gamma$ with respect to some
finite generating set. Giving all edges length one gives 
${\cal C}$ the structure of  
a locally finite geodesic metric graph. 
The group $\Gamma$ acts freely and cocompactly from the left on
${\cal C}$ as a group of isometries. Let $\dist$ be the
$\Gamma$- invariant distance
function on ${\cal C}$. 

The \emph{critical exponent} $\delta(\Gamma)$ of
$\Gamma$ is the infimum of all numbers $s>0$ such that the
\emph{Poincar\'e series} 
\[\sum_{\psi\in \Gamma}e^{-s\,{\rm dist}
(y,\psi x)}\]
converges for some and hence all $x,y\in {\cal C}$.
Since vertices of ${\cal C}$ of distance $\ell$ to the identity correspond to 
reduced words in the generating set of length $\ell$, 
%finitely generated,
%and of exponential growth, 
the critical exponent of $\Gamma$ is finite.

\begin{lemma}\label{bound}
For any $s>\delta(\Gamma)$ the value of the convergent series
\[\sum_{\psi\in \Gamma}e^{-s\, {\rm dist}(y,\psi x)}\] is bounded
independently of $x,y\in {\cal C}$.
\end{lemma}
\begin{proof}
For $\zeta\in \Gamma$ we have
\[\sum_{\psi\in \Gamma}e^{-s\,{\rm dist}(\zeta y,\psi x)}=
\sum_{\psi\in \Gamma}e^{-s\,{\rm dist}(\zeta y,\zeta\psi x)}=
\sum_{\psi\in \Gamma}e^{-s\,{\rm dist}(y,\psi x)}\]
and hence the claim is immediate from continuity of the distance
function and cocompactness of the action of $\Gamma$ on ${\cal C}$.
\end{proof}

Let us now assume that the group $\Gamma$ acts as a group of 
homeomorphisms on the compact metrizable space 
$X$. Choose a metric $\hat \delta$ on $X$ defining its topology. 
By compactness, the diameter of $\hat \delta$ is finite.
View the identity $e$ of $\Gamma$ 
as a  basepoint in $\Gamma\subset {\cal C}$.
% and put $\hat \delta_e=\hat \delta$.
%and choose a 
%(finite diameter) distance
%$\hat\delta_e$ on $X$ defining the topology of $X$. 
For $\psi\in \Gamma$ write 
\[\hat\delta_{\psi}=\hat\delta_e\circ \psi^{-1};\] this defines
a $\Gamma$-equivariant family of distance functions on 
$X$ indexed by the elements of $\Gamma$, with 
$\hat \delta_e=\hat \delta$.
Let $s>\delta(\Gamma)$ and
for $p\in {\cal C}$ define 
\begin{equation}\label{delta}
\delta_p=\sum_{\psi\in \Gamma} e^{-s\,{\rm dist}(p,\psi)}\hat\delta_{\psi};\end{equation}
as the sum of two distance functions is a distance function, in view of Lemma \ref{bound} 
this defines a distance function on $X$. 

The following proposition summarizes the properties of these distance
functions we are interested in.

\begin{proposition}\label{local}
\begin{enumerate}
\item 
The distances $\delta_p$ $(p\in {\cal C})$
are mutually bi-Lipschitz equivalent, with bi-Lipschitz constant
bounded by a continuous $\Gamma$-invariant function on 
${\cal C}\times {\cal C}$, and
they define the original topology on $X$.
\item
For all $p\in {\cal C},\psi\in \Gamma$ 
we have
$\delta_{\psi p}=\delta_p\circ \psi^{-1}$. Furthermore, each $\psi\in \Gamma$
acts on 
$(X,\delta_e)$ as a bi-Lipschitz transformation.
%\item 
%If the action of $\Gamma$ on $(B,\hat\delta_e)$ is by
%weakly $H$-quasi-symmetric homeomorphisms, then 
%for each $g$ the map $(B,\hat \delta_g)\to 
%(B,\delta_g)$ is a 
%weak $H$-quasi-symmetry, and the map 
%$g:(B,\delta_e)\to (B,\delta_e)$ is Lipschitz. 
%\item  
%If the action is by
%$\eta$-quasi-symmetric homeomorphisms, then 
%for each $g\in \Gamma$ the map 
%$g:(B,\hat \delta_e)\to (B,\delta_e)$ is $\eta\circ \eta$-quasi-symmetric.
\item Up to adjusting the parameter $s>0$, 
if the action of $\Gamma$ on $(X,\hat \delta)$ is by 
bi-Lipschitz transformations, then the metric $\delta_e$ is bi-Lipschitz
equivalent to $(X,\hat \delta)$.
\end{enumerate}
\end{proposition}
\begin{proof} 
By the definition of the 
distances $\delta_p$, for all $\eta\in \Gamma$ we have
\begin{equation}\label{delta2}
\delta_{p}\circ \eta^{-1}=(\sum_\psi e^{-s\,{\rm dist}(p,\psi)}\hat \delta_{e}\circ \psi^{-1})
\circ \eta^{-1}
=\sum_{\eta \psi} e^{-s\,{\rm dist}(\eta p,\eta \psi)}\hat \delta_{\eta \psi}=\delta_{\eta p}\end{equation}
which shows equivariance.

To show that the distances $\delta_p$ are mutually 
bi-Lipschitz equivalent,   
it suffices to observe that for all $p\in {\cal C}$ we have 
\begin{equation}\label{lipschitzbound}
  e^{-s\, {\rm dist}(p,e)}\delta_e\leq \delta_{p}\leq e^{s \,{\rm dist}(p,e)}\delta_e.
\end{equation}
To this end note that by definition and the triangle inequality, for $p \in {\cal C}$ we have 
\[\delta_p=\sum_{\psi\in \Gamma}e^{-s\, {\rm dist}(p,\psi)}\hat \delta_\psi\geq 
e^{-s\, {\rm dist}(p,e)}\sum_{\psi\in \Gamma} e^{-s\, {\rm dist}(e,\psi)} \hat \delta_\psi,\]
which yields
\[\delta_p\geq e^{-s\, {\rm dist}(p,e)}\delta_e.\]
The reverse estimate $\delta_e\geq e^{-s\, d(p,e)}\delta_p$ follows from 
exactly the same argument.

By formula (\ref{delta2}), 
for all $\eta\in \Gamma$  the map 
$\eta:(X,\delta_e)\to (X,\delta_{\eta})$ is an isometry.
It now follows from bi-Lipschitz equivalence of the metrics
$\delta_\psi$ $(\psi \in \Gamma)$ 
that $\psi:(X,\delta_e)\to (X,\delta_e)$ is bi-Lipschitz, with controlled
bi-Lipschitz constant.
The second part of the proposition follows.

We claim that
the distances $\delta_p$ define the original
topology on $X$. To this end  
note that by definition, we have
$\delta_e\geq \hat\delta_e$ and hence the 
identity map $(X,\delta_e)\to (X,\hat\delta_e)$ is one-Lipschitz.
In particular, the identity map $(X,\delta_e)\to (X,\hat\delta_e)$ 
is continuous. 

To show that the identity map $(X,\hat \delta_e)\to (X,\delta_e)$ is 
continuous as well it suffices to show that for every $x\in X$ and for every 
$\epsilon >0$ the open
ball $B_e(x,\epsilon)$ 
of radius $\epsilon$ about $x$ for the
metric $\delta_e$ contains a neighborhood of $x$ 
for the topology defined by the metric
$\hat\delta_e$. 

To this end 
let $D>0$ be the diameter of $\hat\delta_e$. 
For $\epsilon >0$
there is a finite subset $A\subset \Gamma$ so that 
\begin{equation}\label{sum1}
\sum_{\psi\not\in A}e^{-s\,{\rm dist}(e,\psi)}D<\epsilon/2.\end{equation}

Let $b>0$ be such that $\sum_{\psi\in A}e^{-s\,{\rm dist}(e,\psi)}b<\epsilon/2$ and 
let $C=\cap_{\psi\in A}\hat B_\psi(x,b)$ where $\hat B_\psi(x,b)$ denotes 
the open ball of radius $b$ about $x$ for the metric $\hat \delta_{\psi}$.
Note that $C$ is an open neighborhood of $x$ 
for the topology induced by the metric $\hat \delta_e$ 
because the group $\Gamma$
acts on $(X,\hat \delta_e)$ as a group of homeomorphisms and $A$ is finite.

If $y\in C$ then $\hat \delta_\psi(x,y)<b$ for all $\psi\in A$ and hence
$\sum_{\psi\in A}e^{-s\, {\rm dist}(e,\psi)}\hat \delta_\psi(x,y)<\epsilon/2$.
Together with (\ref{sum1}), this yields 
$\delta_e(x,y)<\epsilon$. 
 As $C\subset X$ is open for the topology defined by $\hat \delta_e$, 
this implies that the ball of radius $\epsilon$ about $x$ for the metric
$\delta_e$ contains an open neighborhood of $x$ for $\hat \delta_e$.
Since $\epsilon >0$ was arbitrary, we conclude that 
a subset of $X$ which is open for 
the topology induced by $\delta_e$ also is open for the 
topology induced by $\hat \delta_e$. In other words,  the 
identity $(X,\hat \delta_e)\to (X,\delta_e)$ is indeed continuous.
This completes the proof of part (1) of the proposition. 

To show the third part of the proposition, 
we have to show that if $\Gamma$ acts by bi-Lipschitz transformations then 
up to adjusting $s$, the identity 
$(X,\hat \delta)\to (X,\delta_e)$ is Lipschitz. To this end 
let $\psi_1,\dots,\psi_k$ be the
symmetric generating set defining the Cayley graph ${\cal C}$. 
Assume that the action of $\Gamma$ on $(X,\hat \delta)$ is by
bi-Lipschitz transformations and 
let $L_i\geq 1$ be the bi-Lipschiitz constant of the element $\psi_i$.
Write $L=\max\{L_i\mid i\}$. Then the bi-Lipschitz constant 
of any $\psi\in \Gamma$ does not exceed $L^{{\rm dist}(e,\psi)}$. 

Now assume that $s>0$ is large enough that $e^s>e^{\delta}L$ where
as before, $\delta>0$ is the critical exponent of $\Gamma$, say
$e^{-s}L\leq e^{-u}$ for some $u>\delta$.  
Then for all $x,y\in X$ and all $\psi\in \Gamma$ we have
$e^{-s\,{\rm dist}(e,\psi)}L^{{\rm dist}(e,\psi)}<e^{-u\,{\rm dist}(e,\psi)}$.
As a consequence, the identity 
$(X,\hat \delta)\to (X,e^{-s\,{\rm dist}(e,\psi)}\hat \delta_\psi)$ is 
$e^{-u\, {\rm dist}(e,\psi)}$-Lipschitz. Summing over $\psi\in \Gamma$ then 
yields that the identity 
$(X,\hat \delta)\to (X,\delta_e)$ is $\sum_\psi e^{-u\,{\rm dist}(e,\psi)}$-Lipschitz.
Since this sum converges, Lipschitz equivalence of 
$(X,\hat \delta)$ and $(X,\delta_e)$ follows.  
\end{proof}

\begin{remark} 
The proof of Theorem \ref{main} rests on the existence of a left invariant
metric on the group $\Gamma$ whose growth is bounded from above
by an exponential function. Although the Birkhoff Katukani theorem 
gives a left invariant metric on any countable group, this metric may not 
fulfill a growth condition. We do not know whether Theorem \ref{main} holds
true for countable infinitely generated groups. 
\end{remark}

As an easy consequence we obtain the proof of Theorem \ref{deroin} 
for finitely generated groups.

\begin{corollary}\label{circlemap}
A finitely generated group of homeomorphisms of $S^1$ is conjugate to 
a group of Lipschitz homeomorphisms.
\end{corollary} 
\begin{proof}
For a fixed basepoint on $S^1$ and the choice of an orientation, 
a volume normalized length metric on $S^1$ is just a Borel 
probability measure $\mu$ on $S^1$ of 
full support and with no atoms. 
The standard normalized Lebesgue measure $\lambda$ on 
$S^1$ corresponds to the standard distance function.

A probability measure $\mu$ on $S^1$ of full support without atoms
defines a homeomorphism $\Psi_\mu:S^1\to S^1$ by
$\Psi_\mu(t)=s$ if $\mu[0,s]=t$. This homeomorphism satisfies 
$(\Psi_\mu)_*\lambda =\mu$. 

Starting with the standard distance $d$ and the  
Lebesgue measure $\lambda$ on $S^1$, 
the distance function $\delta$ constructed in the proof of 
Proposition \ref{local} corresponds to the measure 
\[\nu= \sum_\psi e^{-s\,{\rm dist}(\psi, e)} \psi_*\lambda.\]
Putting $\mu=\nu/\nu(S^1)$, the conjugation 
of the action of $\Gamma$ on $(S^1,d)$ with a Lipschitz action 
is given by the homeomorphism $\Psi_\mu$ considered in 
the previous paragraph.
\end{proof}

\noindent
MATHEMATISCHES INSTITUT DER UNIVERSIT\"AT BONN\\
email: ursula@math.uni-bonn.de

\end{document}